\documentclass{daj}
\usepackage{amsmath,amsthm,amssymb,xcolor,verbatim,graphicx,fullpage,url,cleveref,soul,stackengine}

\newtheorem{theorem}{Theorem}[section]

\newtheorem{lemma}[theorem]{Lemma}

\newtheorem{remark}[theorem]{Remark}

\newcommand{\F}{\mathbb{F}}
\newcommand{\hr}{\textup{h}}
\newcommand{\vr}{\textup{v}}
\newcommand{\LL}{\mathcal{L}}

\newcommand{\s}{\mathbf{s}}
\newcommand{\Img}{\text{Im}}
\newcommand{\eps}{\varepsilon}

\newcommand{\codim}{\text{codim}}
\newcommand{\proj}{\textrm{Proj}}
\newcommand{\rank}{\textrm{rank}}
\newcommand{\OR}{\textrm{OR}}
\newcommand{\AND}{\textrm{AND}}
\newcommand{\poly}{\textrm{poly}}

\dajAUTHORdetails{%
  title = {A bilinear Bogolyubov-Ruzsa lemma with polylogarithmic bounds}, 
  author = {Kaave Hosseini and Shachar Lovett},
  plaintextauthor = {Kaave Hosseini, Shachar Lovett},
  keywords = {Additive combinatorics, Bogolyubov-Ruzsa lemma, bilinear set.},
}   

\dajEDITORdetails{%
   year={2019},
   volume={XX},
   number={10},
   received={30 August 2018},   
   revised={21 May 2019},    
   published={14 June 2019},  
   doi={10.19086/daXXX},       
}   

\begin{document}

\begin{frontmatter}[classification=text]

\title{A bilinear Bogolyubov-Ruzsa lemma with polylogarithmic bounds} 

\author[kaave]{Kaave Hosseini\thanks{Supported by NSF grant CCF-1614023.}}
\author[shachar]{Shachar Lovett\thanks{Supported by NSF grant CCF-1614023.}}

\begin{abstract}
The Bogolyubov-Ruzsa lemma, in particular the quantitative bound obtained by Sanders, plays a central role
in obtaining effective bounds for the $U^3$ inverse theorem for the Gowers norms. Recently, Gowers and Mili\'cevi\'c
applied a bilinear Bogolyubov-Ruzsa lemma as part of a proof of the $U^4$ inverse theorem
with effective bounds.
The goal of this note is to obtain a quantitative bound for the bilinear Bogolyubov-Ruzsa lemma which is similar to
	that obtained by Sanders for the Bogolyubov-Ruzsa lemma.

We show that if a set $A \subset \F^n \times \F^n$
has density $\alpha$, then after a constant number of horizontal and vertical sums, the set $A$ contains a bilinear
structure of codimension $r=\log^{O(1)} \alpha^{-1}$. This improves
the result of Gowers and Mili\'cevi\'c, who obtained a similar statement with a weaker bound of
 $r=\exp(\exp(\log^{O(1)} \alpha^{-1}))$, and by Bienvenu and L\^e, who obtained $r=\exp(\exp(\exp(\log^{O(1)} \alpha^{-1})))$.
\end{abstract}
\end{frontmatter}

\section{Introduction}

One of the key ingredients in the proof of the quantitative inverse theorem for the Gowers $U^3$ norm over finite fields, due to Green and Tao~\cite{green2008inverse} and Samorodnitsky~\cite{samorodnitsky2007low}, is an inverse theorem concerning the structure of sumsets. The best known result in this direction is the quantitatively improved Bogolyubov-Ruzsa lemma due to Sanders~\cite{sanders2012bogolyubov}.
Before introducing it, we fix some common notation.
We assume that $\F=\F_p$ is a prime field, where $p$ is a fixed constant, and suppress the exact dependence on $p$ in the bounds. Given
a subset $A \subset \F^n$ its density is $\alpha = |A|/|\F|^n$. The sumset of $A$ is $2A=A+A=\{a+a': a,a' \in A\}$, and its difference set
is $A-A=\{a-a': a,a' \in A\}$.

\begin{theorem}(\cite{sanders2012bogolyubov})\label{theorem:sanders}
Let $A \subset \F^n $ be a subset of density $\alpha$. Then there exists a subspace $V$ of $\F^n $ of codimension $O(\log^4 \alpha^{-1})$ such that $V\subset 2A-2A$ .
\end{theorem}

In fact, the link between the $U^3$ inverse theorem and inverse sumset results runs deeper. It was shown in \cite{green2010equivalence,lovett2012equivalence} that a $U^3$ inverse theorem with (to date conjectural) polynomial bounds is equivalent to the polynomial Freiman-Ruzsa conjecture, one of the central open problems in additive combinatorics.
Given this, one cannot help but wonder whether there is a more general inverse sumset phenomenon that would naturally correspond to quantitative inverse theorems for the $U^k$ norms. In a recent breakthrough, Gowers and Mili\'cevi\'c~\cite{gowers2017quantitative} showed that this is indeed the case, at least for the $U^4$ norm. They used a \textit{bilinear} generalization of \Cref{theorem:sanders} to obtain a quantitative $U^4$ inverse theorem.

To be able to explain this result we need to introduce some notation. Let $A \subset \F^n \times \F^n$. Define two operators, capturing subtraction on horizontal and vertical fibers as follows:
\begin{align*}
&\phi_\hr(A) := \{(x_1-x_2, y): (x_1,y),(x_2,y) \in A\},\\
&\phi_\vr(A) := \{(x, y_1-y_2): (x,y_1),(x,y_2) \in A\}.
\end{align*}
Given a word $w \in \{\hr,\vr\}^k$ define $\phi_w = \phi_{w_1} \circ \ldots \circ \phi_{w_k}$ to be their composition.
A \emph{bilinear variety} $B \subset \F^n \times \F^n$ of codimension $r=r_1+r_2+r_3$ is a set defined as follows:
$$
B = \{(x,y) \in V \times W: b_1(x,y)=\ldots=b_{r_3}(x,y)=0\},
$$
where $V, W \subset \F^n$ are subspaces of codimension $r_1,r_2$, respectively,
and $b_1,\ldots,b_{r_3}:\F^n \times \F^n \to \F$ are bilinear forms.

Gowers and Mili\'cevi\'c~\cite{gowers2017bilinear} and independently Bienvenu and L\^e~\cite{bienvenu2017bilinear} proved the following, although~\cite{bienvenu2017bilinear} obtained a weaker bound of $r = \exp(\exp(\exp(\log^{O(1)}\alpha^{-1})))$.

\begin{theorem}[\cite{gowers2017bilinear,bienvenu2017bilinear}]
\label{theorem:gowersbilinear}
Let $A \subset \F^n \times \F^n$ be of density $\alpha$ and let $w = \hr\hr\vr\vr\hr\hr$. Then there exists a bilinear variety $B \subset \phi_w(A)$
of codimension $r=\exp(\exp(\log^{O(1)}\alpha^{-1}))$.
\end{theorem}

To be precise, it was not \Cref{theorem:gowersbilinear} directly but a more analytic variant of it that was used (combined with many other ideas) to prove the $U^4$  inverse theorem in \cite{gowers2017quantitative}. However, we will not discuss this analytic variant here.

The purpose of this note is to improve the bound in \Cref{theorem:gowersbilinear} to $r=  \log^{O(1)} \alpha^{-1}$, as was conjectured in \cite{bienvenu2017bilinear}. Our proof is arguably simpler and is obtained only by invoking \Cref{theorem:sanders} a few times, without doing any extra Fourier analysis. The motivation behind this work --- other than obtaining a bound that matches the linear case --- is to employ this result in a more algebraic framework to obtain a modular and simpler proof of a $U^4$ inverse  theorem. Moreover, \Cref{theorem:main} was recently used by Bienvenu and L\^e~\cite{bienvenu2019linear} to obtain upper bounds on the correlation of the M\"obius function over functions fields with quadratic phases.

One more remark before stating the result is that \Cref{theorem:gowersbilinear} generalizes \Cref{theorem:sanders} because given a set $A\subset\F^n$, one can apply \Cref{theorem:gowersbilinear} to the set $A' = \F^n\times A$ and find $\{x\}\times V \subset \phi_w(A')$, where $x$ is arbitrary, and $V$ a subspace of codimension $3r$. This implies $V\subset 2A-2A$.

\begin{theorem}[\textbf{Main theorem}]\label{theorem:main}
Let $A \subset \F^n \times \F^n$ be of density $\alpha$ and let $w = \hr\vr\vr\hr\vr\vr\vr\hr\hr$. Then there exists a bilinear variety $B \subset \phi_w(A)$
of codimension $r=O(\log^{80}\alpha^{-1})$.
\end{theorem}

Note that the choice of the word $w$ in \Cref{theorem:main} is $w = \hr\vr\vr\hr\vr\vr\vr\hr\hr$, which is slightly longer than the word $\hr\hr\vr\vr\hr\hr$ in \Cref{theorem:gowersbilinear}. However, for applications this usually does not matter and any constant length $w$ would do the job. In fact, allowing $w$ to be longer is what enables us to obtain a result with a stronger bound.

\subsection{A robust analog of \Cref{theorem:main}}
Returning to the theorem of Sanders, there is a more powerful variant of \Cref{theorem:sanders} which guarantees that $V$ enjoys a stronger property than just being a subset of $2A-2A$. The stronger property is that every element $y\in V$ can be written in many ways as $y = a_1+a_2-a_3-a_4$, with  $a_1,a_2,a_3,a_4 \in A$. This stronger property of $V$ has a number of applications, for example to upper bounds for Roth's theorem in four variables. We refer the reader to \cite{schoen2016roth}, where Theorem 3.2 is obtained from \Cref{theorem:sanders} and the aforementioned application is given.

\begin{theorem}[\cite{sanders2012bogolyubov,schoen2016roth}]
\label{theorem:statsanders}
Let $A \subset \F^n $ be a subset of density $\alpha$. Then there exists a subspace $V\subset 2A-2A$ of codimension $O(\log^4 \alpha^{-1})$ such that the following holds. Every $y \in V$ can be expressed as $y=a_1+a_2-a_3-a_4$ with $a_1,a_2,a_3,a_4 \in A$ in at
least $\alpha^{O(1)} |\F|^{3n}$ many ways.
\end{theorem}

In \Cref{section:main2} we also prove a bilinear version of \Cref{theorem:statsanders} by slightly modifying the proof of \Cref{theorem:main}. To explain it, we need just a bit more notation.

 Fix an arbitrary $(x,y)\in \F^n\times \F^n$, and note that $(x,y)$ can be written as $(x,y) = \phi_\hr((x+x_1,y),(x_1,y))$ for any $x_1\in \F^n$. Moreover, for any fixed $x_1$, each of the points $(x+x_1,y), (x_1,y)$ can be written as $(x+x_1,y) = \phi_\vr((x+x_1,y+y_1),(x+x_1,y_1))$ and $(x_1,y) = \phi_\vr((x_1,y+y_2),(x_1,y_2))$ for arbitrary $y_1,y_2\in \F^n$. So overall, the point $(x,y)$ can be written  using the operation $\phi_{\vr\hr}$ in exactly $|\F^n|^3$ many ways, namely,  the total number of two-dimensional parallelograms $(x+x_1,y+y_1),(x+x_1,y_1),(x_1,y+y_2),(x_1,y_2)$, where $(x,y)$ is fixed. More generally, for an arbitrary word $w\in \{\hr,\vr\}^k$, the point $(x,y)$ can be written using the operation $\phi_w$ in exactly $|\F^n|^{2^k-1}$ many ways.

Given a set $A\subset \F^n \times \F^n$ and a word $w\in \{\hr,\vr\}^k$, we define $\phi_w^\eps(A)$ to be the set of all elements $(x,y)\in  \F^n\times \F^n$ that can be obtained in at least $\eps|\F^n|^{2^k-1}$ many ways by applying the operation $\phi_w(A)$.


The following is an extension of \Cref{theorem:main} similar in spirit to \Cref{theorem:statsanders}.
\begin{theorem}\label{theorem:main2}
	Let $A \subset \F^n \times \F^n$ be of density $\alpha$ and let $w = \hr\vr\vr\hr\vr\vr\vr\hr\hr$ and $\eps = \exp(-O(\log^{20}\alpha^{-1}))$. Then there exists a bilinear variety $B \subset \phi^{\eps}_w(A)$
of codimension $r=O(\log^{80}\alpha^{-1})$.
\end{theorem}

As a final comment, we remark that if one keeps track of the dependence on the size of the field throughout the proofs, then the bound in \Cref{theorem:main} and \Cref{theorem:main2}
is  $r = O(\log^{80}\alpha^{-1} \cdot \log^{O(1)} |\F|)$.

\paragraph*{Paper organization.} We prove \Cref{theorem:main} in \Cref{section:main} and \Cref{theorem:main2} in \Cref{section:main2}.

\section{Proof of \Cref{theorem:main}}
\label{section:main}

We prove \Cref{theorem:main} in six steps, which correspond to applying the chain of operators $\phi_{\hr} \circ \phi_{\vr\vr} \circ \phi_\hr \circ \phi_\vr \circ \phi_{\vr\vr}\circ \phi_{\hr\hr}$ to $A$. In the proof, we invoke \Cref{theorem:sanders} (or \Cref{theorem:statsanders}, or the Freiman-Ruzsa theorem, which is a corollary of \Cref{theorem:sanders}), four times in total, in Steps 1, 2, 4, and 5.

We will assume that $A \subset \F^m \times \F^n$, where initially $m=n$ but where throughout the proof we update $m,n$ independently when we restrict $x$ or $y$ to large subspaces. It also helps readability, as we will always have that $x$ and related sets or subspaces are in $\F^m$,
while $y$ and related sets or subspace are in $\F^n$.

We use three variables $r_1,r_2,r_3$ that hold the total number of linear forms on $x$, linear forms on $y$, and bilinear forms on $(x,y)$ that are being fixed throughout the proof, respectively. Initially, $r_1=r_2=r_3=0$, but their values will be updated as we go along and at the end $r = r_1+r_2+r_3$ will be the codimension of the final bilinear variety.

\paragraph{Step 1.}
Decompose $A = \bigcup_{y \in \F^n} A_y \times \{y\}$ with $A_y \subset \F^m$. Define $A^1 := \phi_{\hr\hr}(A)$, so that
$$
A^1 = \bigcup_{y \in \F^n} (2A_y - 2A_y) \times \{y\}.
$$
Let $\alpha_y$ denote the density of $A_y$. By \Cref{theorem:sanders},
there exists a linear subspace $V'_y \subset 2 A_y - 2A_y$ of codimension $ O(\log^4\alpha_y^{-1})$.
Let $S:=\{y: \alpha_y \ge \alpha/2\}$, where by averaging $S$ has density $\ge \alpha/2$.
Note that for every $y \in S$ the codimension of each $V'_y$ is $O(\log^4\alpha^{-1})$.
We have
$$
B^1 := \bigcup_{y \in S} V'_y \times \{y\} \subset A^1.
$$

\paragraph{Step 2.}
Consider $A^2:=\phi_{\vr\vr}(B^1)$. It satisfies
$$
A^2=\bigcup_{y_1,y_2,y_3,y_4 \in S} \left( V'_{y_1} \cap V'_{y_2} \cap V'_{y_3} \cap V'_{y_4} \right) \times \{y_1+y_2-y_3-y_4\}.
$$
By \Cref{theorem:sanders}, there is a subspace $W' \subset 2S - 2S$ of codimension $O(\log^4\alpha^{-1})$. Note that the codimension
of $W'$, as well as the codimension of each $V'_{y_1} \cap V'_{y_2} \cap V'_{y_3} \cap V'_{y_4}$, is at most $O(\log^4\alpha^{-1})$.
We thus have
$$
B^2:=\bigcup_{y \in W'} V_y \times \{y\} \subset A^2,
$$
where $V_y = V'_{y_1} \cap V'_{y_2} \cap V'_{y_3} \cap V'_{y_4}$ for some $y_1,y_2,y_3,y_4 \in S$ which satisfy $y=y_1+y_2-y_3-y_4$.

Update $r_2 := \codim(W')$, where we restrict $y \in W'$. To simplify the notation, identify $W' \cong \F^{n-\codim(W')}$
and update $n := n - \codim(W')$. Thus we assume from now that
$$
B^2:=\bigcup_{y \in \F^n} V_y \times \{y\},
$$
where each $V_y$ has codimension $d=O(\log^4\alpha^{-1})$.

\paragraph{Step 3.}
Consider $A^3:=\phi_{\vr}(B^2)$. It satisfies
$$
A^3=\bigcup_{y,z \in \F^n} \left( V_{z} \cap V_{y+z} \right) \times \{y\}.
$$

\paragraph{Step 4.}
Consider $A^4 := \phi_{\hr}(A^3)$. It satisfies
$$
A^4=\bigcup_{y,z,w\in \F^n}\left( \left( V_{z} \cap V_{y+z} \right) + \left( V_{w} \cap V_{y+w} \right) \right ) \times \{y\}.
$$
Define $U_y := V_y^\perp$, so that $\dim(U_y)=d$ and
$$
A^4=\bigcup_{y,z,w \in \F^n} \left( \left( U_{z} + U_{y+z} \right) \cap \left( U_{w} + U_{y+w} \right) \right )^{\perp} \times \{y\}.
$$
We pause for a moment to introduce one more useful piece of notation.
Recall that an affine map $L:\F^n \to \F^m$ is of the form $L(y)=My+b$ where $M \in \F^{m \times n}, b \in \F^m$.
Given a set of affine maps $\LL = \{L_i:\F^n\rightarrow \F^m, i\in [k]\}$ and $y\in \F^n$, let $\LL(y) = \{ L_1(y),\dots,L_k(y)\} \subset \F^m$, and let $\overline{\LL}$ denote the linear span of $\LL$.
Our goal in this step is to find a small family of affine maps $\LL$ with $|\LL|=O(d)$, and a fixed choice of $z,w$, so that

\begin{equation}\label{equation:step4goal}
\Pr_{y\in \F^n} \bigg[ \left( U_{z}+ U_{y+z} \right) \cap \left( U_{w} + U_{y+w} \right) \subset \overline{\LL}(y)  \bigg]   \gg 1,
\end{equation}
as this will give us a dense set $T\subset \F^n$ so that
$$
\bigcup_{y \in T} \overline{\LL}(y)^\perp \times \{y\} \subset A^4.
$$
We now explain how to get \Cref{equation:step4goal}. For every $a\in \F^n$, let $\LL_a$ be a collection of affine maps where initially $\LL_a = \{0\}$ for all $a$'s. We keep adding affine maps to some of the $\LL_a$s, while always maintaining $|\LL_a|\leq 2^d$ for all $a\in \F^n$, until we satisfy
\begin{equation}\label{equation:step4beforegoal}
	\Pr_{y,z,w\in \F^n} \bigg[ \left( U_{z}+ U_{y+z} \right) \cap \left( U_{w} + U_{y+w} \right) \subset \overline{\LL}_z(z) + \overline{\LL}_{y+z}(y+z) + \overline{\LL}_w(w) +\overline{\LL}_{y+w}(y+w)   \bigg] \geq \frac{1}{2}
\end{equation}
and then we will pick some popular affine maps $\LL \subset \cup_{a\in \F^n} \LL_a$ with $|\LL|=O(d)$ that will give us \Cref{equation:step4goal}. For now, we show how to get \Cref{equation:step4beforegoal}.
We need the following lemma.
\begin{lemma}
\label{lemma:intersect}
For each $y \in \F^n$, let $U_y \subset \F^m$ be a subspace of dimension $d$. Assume that
$$
	\Pr_{y,z,w\in \F^n} \bigg[ \left( U_{z}+ U_{y+z} \right) \cap \left( U_{w} + U_{y+w} \right) \subset \overline{\LL}_z(z) + \overline{\LL}_{y+z}(y+z) + \overline{\LL}_w(w) +\overline{\LL}_{y+w}(y+w)   \bigg] \leq \frac{1}{2}.
$$
Then there exists an affine function $L:\F^n \to \F^m$ such that
$$
	\Pr_{y \in \F^n} \left[L(y) \in U_y\setminus \overline{\LL}_y(y)\right] \ge \exp(-O(d^4)).
$$
\end{lemma}
In the following we prove \Cref{lemma:intersect}. We will use a modified version of a \emph{functional version} of the Freiman-Ruzsa theorem, with the quasi-polynomial bounds obtained by Sanders \cite{sanders2012bogolyubov}. We first recall the standard version. For details of how it is derived from \Cref{theorem:sanders} we refer the reader to \cite{green2005notes}. In fact, in this case the bound can be slightly improved. The reader is referred to~\cite{sanders2012structure}.

\begin{theorem}\label{theorem:qpfr}(Freiman-Ruzsa theorem; functional version).
	Let $f:\F^n \rightarrow \F^m$ be a function. Suppose that
	$$\Pr_{y,z,z'\in \F^n}\left[f(y+z)- f(z)=f(y+z')-f(z')\right]\geq \alpha.$$
	Then there exists an affine map $L:\F^n\rightarrow \F^m$ such that
$$
	|\{z \in \F^n: L(z)=f(z)\} |\ge \exp(-O(\log^{4}(\alpha^{-1})))|\F^n|.
$$
\end{theorem}

Now, this result may be strengthened as follows.

\begin{lemma}\label{lemma:strongsanders}
	Let $f:\F^n \rightarrow \F^m$ be a function and $Z\subset\F^n$ with $|Z|\geq \alpha |\F^n|$. Suppose that
	$$\Pr_{y\in \F^n,z,z'\in Z}\left[f(y+z)- f(z)=f(y+z')-f(z')\right]\geq \alpha.$$
	Then there exists an affine map $L:\F^n\rightarrow \F^m$ such that
$$
	|\{z \in Z: L(z)=f(z)\} |\ge \exp(-O(\log^{4}(\alpha^{-1})))|\F^n|.
$$
\end{lemma}

\begin{proof}
	Let $\Gamma = \{(x,f(x)): x\in \F^n\}$ and $\Gamma' = \{(x,f(x)): x\in Z\}$. The additive energy $E(\Gamma,\Gamma')$ is defined as 
$$E(\Gamma,\Gamma') = |\{(a,b,c,d):a-b = c-d, a,c\in \Gamma, b,d\in \Gamma' \}|$$
and satisfies
$$E(\Gamma,\Gamma') \geq \alpha^{O(1)} |\Gamma|^3.$$
Using the Cauchy-Schwarz inequality for additive energy (see Corollary 2.10 in \cite{tao2006additive}), we have
	$$E(\Gamma,\Gamma')\leq \sqrt{E(\Gamma,\Gamma)\cdot E(\Gamma',\Gamma')}.$$
	Using the fact that $|\Gamma'|\geq \alpha |\Gamma|$, we get that $E(\Gamma',\Gamma')\geq \alpha^{O(1)}|\Gamma|^3$. Let $M \ge m$ be large enough, and define a function $f':\F^n \to \F^M$ by setting $f'(z)=f(z)$ if $z \in Z$, and otherwise $f$ takes random values in $\F^M$. Apply \Cref{theorem:qpfr} to $f'$. The linear function $L$ thus obtained has to necessarily agree with $f'$ (and hence with $f$) on a subset $Z'\subset Z$ of the claimed density.
\end{proof}
\begin{remark}
	Note that using a bootstrapping argument due to Konyagin, the bound  $\exp(-O(\log^{4}(\alpha^{-1})))|\F^n|$  in \Cref{theorem:qpfr} can be improved to $\exp(-O(\log^{3+o(1)}(\alpha^{-1})))|\F^n|$ (see Theorem 12.5 in ~\cite{sanders2012structure}). Here, we have used the exponent $4$ instead of $3+o(1)$ for aesthetic reasons. Using \Cref{lemma:strongsanders} with exponent $3+o(1)$ in what follows would result in a final bound of $\log^{64+o(1)} \alpha^{-1}$ instead of $\log^{80} \alpha^{-1}$ in \Cref{theorem:main} and \Cref{theorem:main2}.
\end{remark}
Now we may return to the proof of \Cref{lemma:intersect}.
\begin{proof}[Proof of \Cref{lemma:intersect}]
	Consider a choice of $y,w, z$ for which
	$$\left( U_{y+z}+ U_{z}\right) \cap \left(  U_{y+w}+U_{w} \right) \not\subset \overline{\LL}_{y+z}(y+z) +\overline{\LL}_z(z) +\overline{\LL}_{y+w}(y+w)+\overline{\LL}_w(w).$$
	This directly implies that there is an ordered quadruple $(a,b,c,d)$ so that $a\in  U_{y+z}, b\in U_{z}, c\in  U_{y+w}, d\in U_{w}$ with $a-b = c-d\neq 0$ and
	$$	\left(\left[a\notin \overline{\LL}_{y+z}(y+z)\right]
		\OR
		\left[b\notin \overline{\LL}_{z}(z)\right]\right)
	\AND
		\left(\left[c\notin\overline{\LL}_{y+w}(y+w)\right]
		\OR
		\left[d\notin \overline{\LL}_{w}(w)\right]\right).$$
	Consider all the possible solutions of the above formula, namely:
	\begin{itemize}
		\item $\left[a\notin \overline{\LL}_{y+z}(y+z)\right]
			\AND
			\left[c\notin\overline{\LL}_{y+w}(y+w)\right]$
		\item $\left[b\notin \overline{\LL}_{z}(z)\right]
			\AND
			\left[c\notin\overline{\LL}_{y+w}(y+w)\right]$
		\item $\left[a\notin \overline{\LL}_{y+z}(y+z)\right]
			\AND
			\left[d\notin\overline{\LL}_{w}(w)\right]$

		\item $\left[b\notin \overline{\LL}_{z}(z)\right]
			\AND
			\left[d\notin\overline{\LL}_{w}(w)\right]$
	\end{itemize}
    One of these cases occur for at least $1/4$ of the choices of $y,w,z$; assume without loss of generality that it is the last one. The other cases are analogous.

	Next, sample a random function $f:\F^n \to \F^m$ by picking $f(x) \in U_x$ uniformly and independently
	for each $x \in \F^n$. Note that the quadruple $a,b,c,d$ depends on $y,w,z$, and that for each such choice
    $$
    \Pr_f[f(y+z)=a,f(z) = b,f(y+w)=c,f(w)=d] \ge |\F|^{-4d}.
    $$
    Note that when this event happens, by construction we have $f(y+z)-f(z)= f(y+w)-f(w)$. Combining this with the assumption of the lemma, we get
$$
	\Pr_{y,z,w \in \F^n, f} \left[ f(y+z)- f(z) =f(y+w)-f(w), f(z)\in U_z\setminus \overline{\LL_z}(z),f(w)\in U_w\setminus \overline{\LL_w}(w) \right] \ge \frac{1}{2}\cdot \frac{1}{4} \cdot |\F|^{-4d}.
$$
	Fix $f$ where the above bound holds. Let $Z = \{z:f(z)\in U_z\setminus \overline{\LL_z}(z)\}$.  Then, supressing the dependence on the size of the field,
we have  $|Z| \ge \exp(-O(d)) |\F|^n$ and
	$$\Pr_{y\in \F^n, z,w\in Z}\left[ f(y+z)- f(z) =f(y+w)-f(w)\right]\geq \exp(-O(d)).$$
	 By \Cref{lemma:strongsanders}, there exists an affine map $L:\F^n \to \F^m$ and a set $Z'\subset Z$ with $|Z'|\geq \exp(-O(d^4))|\F^n|$ such that for all $z'\in Z'$, $f(z') = L(z')$ and hence $L(z')\in U_{z'}\setminus\overline{\LL_{z'}}(z')$.
\end{proof}

Next, we proceed as follows. As long as \Cref{equation:step4beforegoal} is satisfied, apply \Cref{lemma:intersect} to find an affine map $L:\F^n \to \F^m$. For every
$x$ that satisfies $L(x) \in U_x \setminus \overline{\LL_{x}}(x) $, add the map $L$ to $\LL_x$.
This process needs to stop after $t=\exp(O(d^{4}))$ many steps.
Let $L_1,\ldots,L_t:\F^n \to \F^m$ be the affine maps obtained in this process.
Using this notation, set $\LL' =\cup_{x\in \F^n} \LL_x $.
For every subspace $U_x$, there is a set $\LL'_x \subset \LL'$ of size $|\LL'_x| \leq d$ such that
$$\overline{\LL_x}(x) \subset \overline{\LL'_x}(x).$$
This implies that
$$
		\Pr_{y,z,w \in \F^n} \left[ \left(\left( U_{z}+ U_{y+z}\right) \cap \left( U_{w}+ U_{y+w} \right) \subset \overline{\LL'_z}(z) + \overline{\LL'_{y+z}}(y+z) + \overline{\LL'_w}(w) +\overline{\LL'_{y+w}}(y+w)\right) \right] \ge \frac{1}{2}.
$$
Consider the most popular quadruple $\LL'_1,\LL'_2,\LL'_3,\LL'_4\subset \LL'$ so that
$$
\Pr_{y,z,w \in \F^n} \left[ \left(\left( U_{z}+ U_{y+z}\right) \cap \left( U_{w}+ U_{y+w} \right) \subset \overline{\LL'_1}(z) + \overline{\LL'_{2}}(y+z) + \overline{\LL'_3}(w) +\overline{\LL'_{4}}(y+w)\right) \right] \ge \frac{1}{2} \times {t \choose d}^{-4}.
$$
Let $\LL := \LL'_1 \cup \LL'_2 \cup \LL'_3 \cup \LL'_4$. Recall that $t=\exp(O(d^{4}))$ and hence ${t \choose d} = \exp(O(d^{5}))$.
We have
$$
\Pr_{y,z,w\in \F^n} \bigg[ \left( U_{z}+ U_{y+z} \right) \cap \left( U_{w} + U_{y+w} \right) \subset \overline{\LL}(z) + \overline{\LL}(y+z) + \overline{\LL}(w) +\overline{\LL}(y+w)   \bigg] \ge \exp(-O(d^{5})).
$$
By averaging, there is some choice of $z,w$ such that,
$$
\Pr_{y\in \F^n} \bigg[ \left( U_{z}+ U_{y+z} \right) \cap \left( U_{w} + U_{y+w} \right) \subset\overline{\LL}(z) + \overline{\LL}(y+z) + \overline{\LL}(w) +\overline{\LL}(y+w)   \bigg]   \ge \exp(-O(d^{5})).
$$
Recall that each $L\in \LL$ is an affine map and that $|\LL|\leq 4d$. Thus, $\overline{\LL}(z) , \overline{\LL}(y+z) , \overline{\LL}(w) ,\overline{\LL}(y+w)\subset \overline{\LL}(y)+ Q$
where $Q \subset \F^m$ is a linear subspace of dimension $O(d)$. We thus have
$$
B^4 := \bigcup_{y\in T} (\overline{\LL}(y)+Q)^\perp \times \{y\} \subset A^4,
$$
where $T \subset \F^n$ has density $\exp(-O(d^{5}))$.

To simplify the presentation, we would like to assume that the maps in $\LL$ are linear maps instead of affine maps,
that is, that they do not have a constant term. This can be obtained by restricting $x$ to the subspace orthogonal
to $Q$ and to the constant term in the affine maps in $\LL$. Correspondingly, we update
$r_1 := r_1 + \dim(Q)+|\LL| = O(d)$.
So from now on we may assume that $\LL$ is defined by $4d$ linear maps, and that
$$
B^4 := \bigcup_{y\in T} {\overline{\LL}(y)}^\perp \times \{y\} \subset A^4,
$$
where $T \subset \F^n$ has density $\exp(-O(d^{5}))$.

\paragraph{Step 5.}
Consider $A^5 := \phi_{\vr\vr}(B^4)$ so that
\begin{align*}
	A^5 & =\bigcup_{y_1,y_2,y_3,y_4 \in T} \left( {\overline{\LL}(y_1)}^\perp \cap{\overline{\LL}(y_2)}^\perp \cap {\overline{\LL}(y_3)}^\perp \cap {\overline{\LL}(y_4)}^\perp \right) \times \{y_1+y_2-y_3-y_4\}.
\end{align*}
 By \Cref{theorem:sanders}  there exists a subspace $W\subset 2T-2T$ of codimension $O(d^{20})$.
However, this time, the conclusion is not strong enough for us, and we need to use \Cref{theorem:statsanders} instead.
The following equivalent formulation of \Cref{theorem:statsanders} will be more convenient for us: there is a subspace $W \subset \F^n$ of codimension $O(\log^4{\alpha^{-1}})$ such that, for each $y \in W$ there is
a set $S_y\subset (\F^n)^3$ of density $\alpha^{O(1)}$, such that for all $(a_1,a_2,a_3)\in S_y$,
$$
\quad a_1,a_2,a_3,a_1+a_2-a_3-y \in A.
$$
Apply \Cref{theorem:statsanders} to the set $T$ to obtain the subspace $W$ and the sets $S_y$. We have
\begin{align*}
	B^5 :=& \bigcup_{y\in W}\left( \bigcup_{(y_1,y_2,y_3)\in S_y} \left(\overline{\LL}(y_1)+ \overline{\LL}(y_2)+ \overline{\LL}(y_3) + \overline{\LL}(y_1+y_2-y_3-y)\right)^\perp  \right) \times \{y\} \subset A^5.
\end{align*}
To simplify the presentation we introduce the notation $\overline{\LL}(y_1,y_2,y_3) := \overline{\LL}(y_1)+ \overline{\LL}(y_2)+ \overline{\LL}(y_3)$.
Next, observe that for any $y,y'\in \F^n$, $\overline{\LL}(y')+\overline{\LL}(y+y')= \overline{\LL}(y')+\overline{\LL}(y)$.
Thus we can simplify the expression of $B^5$ to
\begin{align*}
	B^5 =& \bigcup_{y\in W}\left( \bigcup_{(y_1,y_2,y_3)\in S_y} \left(\overline{\LL}(y_1,y_2,y_3) + \overline{\LL}(y) \right)^\perp \right) \times \{y\},
\end{align*}
which can be re-written as
\begin{align*}
	B^5 =& \bigcup_{y\in W}\left( \bigcup_{(y_1,y_2,y_3)\in S_y} \overline{\LL}(y_1,y_2,y_3) ^\perp \cap  {\overline{\LL}(y)}^\perp \right) \times \{y\}.
\end{align*}


\paragraph{Step 6.} Consider $A^6:= \phi_\hr(B^5)$. It satisfies

\begin{align*}
	A^6 = \bigcup_{y\in W}\left( \bigcup_{\substack{(y_1,y_2,y_3) \in S_y \\(y'_1,y'_2,y'_3) \in S_y}}\overline{\LL}(y_1,y_2,y_3)^\perp \cap {\overline{\LL}(y)}^\perp  + \overline{\LL}(y'_1,y'_2,y'_3)^\perp \cap {\overline{\LL}(y)}^\perp   \right) \times \{y\}.
\end{align*}
In order to complete the proof, we will find a large subspace $V$ such that for every $y\in W $,
$$V\cap {\overline{\LL}(y)}^\perp  \subset \bigcup_{\substack{(y_1,y_2,y_3) \in S_y \\(y'_1,y'_2,y'_3) \in S_y}}\overline{\LL}(y_1,y_2,y_3)^\perp\cap {\overline{\LL}(y)}^\perp  + \overline{\LL}(y'_1,y'_2,y'_3)^\perp\cap {\overline{\LL}(y)}^\perp   .$$
In fact, we will prove something stronger: there is a large subspace $V$ such that for each $y\in W$, there is a choice of $(y_1,y_2,y_3),(y'_1,y'_2,y'_3) \in S_y$ for which
$$V\cap {\overline{\LL}(y)}^\perp \subset  \overline{\LL}(y_1,y_2,y_3)^\perp\cap {\overline{\LL}(y)}^\perp  + \overline{\LL}(y'_1,y'_2,y'_3)^\perp\cap {\overline{\LL}(y)}^\perp  .$$
The following lemma is key.
Given a set $\LL$ of linear maps from $\F^n$ to $\F^m$, let $\dim(\overline{\LL})$ denote the dimension of linear span of $\LL$ as a vector space over $\F$.

\begin{lemma}\label{lemma:mapsubspace}
	Fix $\delta>0$. Let $\LL$ be a set of linear maps from $\F^n$ to $\F^m$ with $\dim(\overline{\LL})=k$.
	Then there is a subspace $Z\subset \F^m$ of dimension at most $k(2k+\log \delta^{-1}+3)$ such that the following holds.
For every subset $S\subset \F^n$ of density at least $\delta$, and arbitrary $y\in \F^n$, at least half the pairs $s,s'\in S$ satisfy
	$$ ( \overline{\LL}(s)+ \overline{\LL}(y) )\cap( \overline{\LL}(s')+ \overline{\LL}(y))\subset Z+ \overline{\LL}(y).  $$
\end{lemma}
\begin{proof}
	The proof is by induction on $\dim(\overline{\LL})$. Consider first the base case of $\dim(\overline{\LL})=1$ and suppose that $\overline{\LL} = \langle L\rangle $ for some map $L$. We consider two cases based on the minimum rank of the maps in $\overline{\LL}$. First suppose that rank of every non-zero map in $\overline{\LL}$ (which is the same as rank of $L$) is bigger than $\log \delta^{-1}+5$.  	
		Fix arbitrary $L_1,L_3\in \overline{\LL}\setminus \{0\}$ and $L_2,L_4\in \overline{\LL} $ and $s,y\in \F^n$ and observe that
	$$\Pr_{s'\in S} \left[ L_1(s)+L_2(y) = L_3(s')+L_4(y) \right] <\frac{|\F|^{-(\log \delta^{-1}+5)}}{\Pr_{s'\in\F^n}\left[s'\in S\right]}\leq |\F|^{-(\log \delta^{-1}+5)}\delta^{-1}.$$
	By applying the union bound over all quadruples $L_1,\cdots,L_4 \in \overline{\LL}$, we obtain that

	$$\Pr_{s,s'\in S} \left[  \left(\overline{\LL}(s)+ \overline{\LL}(y)\right) \cap \left(\overline{\LL}(s')+\overline{\LL}(y)\right) \neq \overline{\LL}(y)\right] \le |\F|^4|\F|^{-(\log \delta^{-1}+5)}\delta^{-1}\leq \frac{1}{2}.$$
	Therefore, we can safely choose $Z=\{0\}$ in the lemma. Now, for the second case, suppose that $\rank(L)\leq \log \delta^{-1}+5$. Let $Z = \Img(L)$.  Then for all $s\in \F^n$,
	$\overline{\LL}(s)\subset \Img(L) = Z$, and so $(\overline{\LL}(s)+\overline{\LL}(y))\cap (\overline{\LL}(s')+\overline{\LL}(y))\subset Z\subset Z+\overline{\LL}(y)$.
	
	Now let $\dim(\overline{\LL}) = k$. First, suppose that $\forall L\in \overline{\LL}$,  $\rank(L) > 4k+\log\delta^{-1} + 1$. Then similar to the base case, for all $y\in \F^n$,
	$$\Pr_{s,s'\in S} \left[  \left(\overline{\LL}(s)+ \overline{\LL}(y)\right) \cap \left(\overline{\LL}(s')+\overline{\LL}(y)\right) \neq \overline{\LL}(y)\right] \leq |\F|^{4k}|\F|^{-(4k+\log\delta^{-1} + 1)}\delta^{-1}\leq \frac{1}{2}.$$
	Otherwise, suppose there is some $L\in \overline{\LL}\setminus \{0\}$ with rank at most $ 4k+\log\delta^{-1} + 1$.
	Let $Y$ be a subspace so that $Y \oplus \Img(L) = \F^m$. Let $\proj_Y:\F^n \to Y$ be the projection map along $\Img(L)$ with $\proj_Y(\Img(L))=0$.
Consider the new family of maps
	$$\LL' =\{\proj_{Y}\circ M: M\in \LL\}.$$
	Note that $\overline{\LL'}$ has dimension $\leq k-1$ because $\proj_{Y}\circ L \equiv 0$ and so by the induction hypothesis, there exists
	a subspace $Z'$ of dimension at most $(k-1)(2(k-1)+\log \delta^{-1}+3)$ such that, for all $y\in \F^n$, for least half the pairs $s,s'\in S$ it holds that
	$$(\overline{\LL'}(s)+ \overline{\LL'}(y)) \cap (\overline{\LL'}(s')+\overline{\LL'}(y))\subset Z' + \overline{\LL'}(y).$$	The above implies that
	$$\proj_{Y}((\overline{\LL}(s)+ \overline{\LL}(y)) \cap (\overline{\LL}(s')+\overline{\LL}(y))) \subset Z'+ \proj_{Y}(\overline{\LL}(y))\subset Z'+ \overline{\LL}(y) + \Img(L),$$
	so we can take $Z = Z'+ \Img(L)$.
\end{proof}

We note that for \Cref{theorem:main} we only need a weaker form of \Cref{lemma:mapsubspace}, which states that at least one pair $s,s' \in S$ exists; however, we will need the stronger version for \Cref{theorem:main2}.

We apply \Cref{lemma:mapsubspace} as follows. Define a new family of linear maps $\LL^*$ from $\F^{3n}$ to $\F^m$
as follows.
	For each $L \in \LL$ define three linear maps $L_i$, $i\in\{1,2,3\}$ by:
$$L_i:(\F^{n})^3\rightarrow\F^{m}, L_i(y_1,y_2,y_3) = L(y_i)$$
	and let $$\LL^* := \{L_i : L\in \LL, i\in [3]\}.$$
	Apply \Cref{lemma:mapsubspace} to the family $\LL^*$ with $\delta  = \exp(-O(d^{5}))$ and obtain a subspace $V\subset \F^m$ of codimension $O(d^2\log(\exp(-O(d^{5})))=O(d^{7})$ so that, for every $S_y\subset (\F^n)^3$ with $y\in W$, there exist $(y_1,y_2,y_3),(y'_1,y'_2,y'_3)\in S_y$ for which
$$V\cap \overline{\LL^*}((y,y,y))^\perp\subset  (\overline{\LL^*}((y_1,y_2,y_3))^\perp\cap \overline{\LL^*}((y,y,y))^\perp) + (\overline{\LL^*}((y'_1,y'_2,y'_3))^\perp\cap \overline{\LL^*}((y,y,y))^\perp).$$
This directly implies that
$$V\cap \overline{\LL}(y)^\perp  \subset  (\overline{\LL}(y_1,y_2,y_3)^\perp\cap \overline{\LL}(y)^\perp) + (\overline{\LL}(y'_1,y'_2,y'_3)^\perp \cap \overline{\LL}(y)^\perp).$$
Define
$$
B^6:= \bigcup_{y\in W}\left( V \cap {\overline{\LL}(y)}^\perp \right) \times \{y\} \subset A^6.
$$
Observe that $B^6$ is a bilinear variety defined by $\codim(V)$ many linear equations on $x$, $\codim(W)$ linear equations on $y$
and $|\LL|$ bilinear equations on $(x,y)$.

To complete the proof we calculate the quantitative bounds obtained. We have $d = O(\log^4 \alpha^{-1})$, where $\alpha$ was the density of the original set $A$, and
\begin{align*}
	&r_1 = O(d)+ \codim(V) = O(d^{7}),\\
	&r_2 = O(d) + \codim(W) = O(d^{20}),\\
&r_3 = |\LL| = O(d).
\end{align*}
Together these give the final bound of $r = r_1+r_2+r_3 = O(\log^{80}\alpha^{-1})$.

\section{Proof of \Cref{theorem:main2}}\label{section:main2}

In this section we prove \Cref{theorem:main2} by slightly modifying the proof of \Cref{theorem:main}.
We point out the necessary modifications to the proof of \Cref{theorem:main}.

\paragraph{Step 1.}
	In this step, we use \Cref{theorem:statsanders} instead of \Cref{theorem:sanders} and directly obtain
\begin{equation}\label{equation:firststep}
	B^1 \subset \phi^{\eps_1}_{\hr\hr}(A)
\end{equation}
	for $\eps_1 = \alpha^{O(1)}$.
\paragraph{Step 2.}
	Similarly in this step as well, using \Cref{theorem:statsanders} instead of \Cref{theorem:sanders} gives
\begin{equation}\label{equation:secondstep}	
	B^2 \subset \phi^{\eps_2}_{\vr\vr}(B^1)
\end{equation}
    with $\eps_2 = \alpha^{O(1)}$.
	Recall that from now on we assume for simplicity of exposition that $B^2 = \bigcup_{y \in \F^n} V_y \times \{y\}$.
\paragraph{Steps 3 and 4.}
This step is slightly different from Steps 1 and 2. Here, we are not able to directly produce some set $B^4$ that would satisfy $B^4 \subset \phi_{\hr\vr}^{\eps_4}(B^2)$. But what we can do is to apply the remaining operation $\phi_{\hr\vr\vr\hr\vr}$ to $B^2$ and obtain the final bilinear structure $B^6$ that satisfies what we want, which is
\begin{equation}\label{equation:thirdstep}	
	B^6 \subset \phi_{\hr\vr\vr\hr\vr}^{\eps_6}(B^2)
\end{equation}
for $\eps_6 = \exp(-\poly \log\alpha^{-1})$.
Combining  \Cref{equation:firststep,equation:secondstep,equation:thirdstep} gives
$$B^6\subset  \phi_{\hr\vr\vr\hr\vr\vr\vr\hr\hr}^{\eps}(A)$$
for $\eps = \exp(-\poly \log\alpha^{-1})$.

We establish \Cref{equation:thirdstep} in the rest of the proof.
Recall that previously we showed that the following holds: there is a set of affine maps $\LL$, with $|\LL|= O(d)$, such that
$$
\Pr_{y,w,z\in \F^n} \bigg[ \left( \overline{\LL}(z) + \overline{\LL}(y+z) + \overline{\LL}(w) +\overline{\LL}(y+w) \right)^\perp \subset \left( V^\perp_{z}\cap V^\perp_{y+z} \right) + \left( V^\perp_{w} \cap V^\perp_{y+w} \right)    \bigg] \ge \exp(-O(d^{5}))
$$
and consequently
$$
\Pr_{y,w,z\in \F^n} \bigg[ \left( \overline{\LL}(y) + \overline{\LL}(z) + \overline{\LL}(w)  \right)^\perp \subset \left( V^\perp_{z}\cap V^\perp_{y+z} \right) + \left( V^\perp_{w} \cap V^\perp_{y+w} \right)    \bigg] \ge \exp(-O(d^{5})).
$$
Recall that $d = O(\log^4 \alpha^{-1})$.
Furthermore, we may assume the maps in $\LL$ are linear (instead of affine) after we update $r_1:= r_1+|\LL| = O(d)$.

In the proof of \Cref{theorem:main} we then fixed one popular choice of $w,z$. However, here we cannot do this, as we need many pairs $w,z$.
Let $T$ be the set of $y$s that satisfy
\begin{equation}\label{equation:step4prob}
	\Pr_{w,z\in \F^n} \bigg[ \left( \overline{\LL}(y) + \overline{\LL}(z) + \overline{\LL}(w)  \right)^\perp \subset \left( V^\perp_{z}\cap V^\perp_{y+z} \right) + \left( V^\perp_{w} \cap V^\perp_{y+w} \right)    \bigg] \ge \exp(-O(d^{5})),
\end{equation}
so $T$ has density $\exp(-O(d^{5}))$.
We deduce something stronger from \Cref{equation:step4prob} but we need to introduce some notation first.

For $A,B\subset\F^n$ let $A-_\eta B$ denote the set of all elements $c\in A-B$ that can be written in at least $\eta|\F^n|$ many ways as $c = a-b$ for $a\in A,b\in B$.
To use this notation, note that if $A,B$ are two subspaces of codimension $k$, then $A - B = A -_\eta B$ for $\eta =\exp(-O(k))$. This is because every element $c\in A-B$ can be written as $c = (a+v)-(b+v)$ where $v$ is an arbitrary element in the subspace $A\cap B$ of codimension at most $2k$. So we can improve \Cref{equation:step4prob} to
\begin{equation}\label{equation:step4probconv}
	\Pr_{w,z\in \F^n} \bigg[ \left( \overline{\LL}(y) + \overline{\LL}(z) + \overline{\LL}(w)  \right)^\perp \subset \left( V^\perp_{z}\cap V^\perp_{y+z} \right) -_\eta \left( V^\perp_{w} \cap V^\perp_{y+w} \right)    \bigg] \ge \exp(-O(d^{5})),
\end{equation}
for $\eta = \exp{(-O(d))}.$
\paragraph{Step 5.}
Similar to before, consider the subspace $W\subset 2T-2T$ of codimension $O(d^{20})$ that is given by \Cref{theorem:statsanders}. This subspace $W$ has the following property: if we fix an arbitrary $y\in W$, sample $y_1,y_2,y_3\in \F^n$ uniformly and independently, and set $y_4 = -y + y_1 + y_2 - y_3$, then with probability at least $\exp(-O(d^{5}))$ we have $y_1,y_2,y_3,y_4 \in T$. This means that if we furthermore sample  $w_1,w_2,w_3,w_4,z_1,z_2,z_3,z_4\in \F^n$ uniformly and independently, then, with probability at least $\exp(-O(d^{5}))$, the following four equations simultaneously hold:
$$	\left( \overline{\LL}(y_i) + \overline{\LL}(z_i) + \overline{\LL}(w_i)  \right)^\perp \subset \left( V^\perp_{z_i}\cap V^\perp_{y_i+z_i} \right) -_\eta  \left( V^\perp_{w_i} \cap V^\perp_{y_i+w_i} \right) \qquad  i=1,\ldots,4.
$$
By computing the intersection of the left-hand side and the right-hand side for each $i=1,\ldots,4$, we obtain
that with probability at least $\exp(-O(d^{5}))$,
\begin{equation}\label{equation:allintersections}
\left( \overline{\LL}(y) + \sum_{i=1}^3 \overline{\LL}(y_i)+ \sum_{i=1}^4 \overline{\LL}(z_i)+ \sum_{i=1}^4 \overline{\LL}(w_i)\right)^\perp  \subset  \bigcap_{i=1}^4 \left( \left( V^\perp_{z_i}\cap V^\perp_{y_i+z_i} \right) -_\eta  \left( V^\perp_{w_i} \cap V^\perp_{y_i+w_i} \right)  \right).
\end{equation}
For a given $y \in \F^n, \s = (y_1,y_2,y_3,w_1,w_2,w_3,w_4,z_1,z_2,z_3,z_4) \in (\F^n)^{11}$, let
$$\mathcal{V}_{y,\s} =  \bigcap_{i=1}^4 \left( \left( V^\perp_{z_i}\cap V^\perp_{y_i+z_i} \right) -_\eta  \left( V^\perp_{w_i} \cap V^\perp_{y_i+w_i} \right)  \right),$$
where we recall that $y_4 = -y + y_1 + y_2 - y_3$. Observe that for any $\s$,
$$
\bigcup_{y \in W} \mathcal{V}_{y,\s} \times \{y\} \subset \phi_{\vr\vr\hr\vr}(B^2).
$$
We rewrite \Cref{equation:allintersections} more compactly as
\begin{equation}\label{equation:allintersectionscompact}
	\Pr_{\s} \left[ \left( \overline{\LL}(y) +\overline{\LL}(\s)\right)^\perp  \subset \mathcal{V}_{y,\s} \right] \ge \exp(-O(d^{5})),
\end{equation}
where we use the notation $\overline{\LL}(\s) =  \sum_{i=1}^3 \overline{\LL}(y_i)+ \sum_{i=1}^4 \overline{\LL}(z_i)+ \sum_{i=1}^4 \overline{\LL}(w_i)$.

\paragraph{Step 6.}
Now we consider the ultimate result of applying the operation $\hr\vr\vr\hr\vr$ to $B^2$. Only the final operation $\hr$ remains to be applied. After doing so, we find a subspace $V\subset\F^m$ of codimension $O(d^7)$ that satisfies the following: for any $y\in W$,  choose  $\mathbf{s_1}, \mathbf{s_2}\in(\F^n)^{11}$ uniformly and independently at random. Then with probability $\exp(-O(d^{5}))$,
$$V\cap \overline{\LL}(y)^\perp \subset \mathcal{V}_{y,\mathbf{s_1}}-_\eta \mathcal{V}_{y,\mathbf{s_2}}.$$
where we recall that $\eta = \exp(-O(d))$.

In order to determine $V$, fix $y\in W$ and let $S_y$ be the set of all tuples $\s = (y_1,y_2,y_3,w_1,w_2,w_3,w_4,z_1,z_2,z_3,z_4)\in (\F^n)^{11}$ that satisfy \Cref{equation:allintersectionscompact}. Note that the density of each $S_y$ is at least $\exp(-O(d^{5}))$. To simplify the notation, denote $\s = (s_1,\ldots,s_{11})$.
We invoke \Cref{lemma:mapsubspace} in a similar way to before.  Define a family $\LL^*$ of linear maps, containing linear maps $L_i$ for each $L\in \LL$ and $i=1,\ldots,11$, where
$$
L_i:(\F^{n})^{11}\rightarrow\F^{m}, L_i(\s) = L(s_i).
$$
Applying \Cref{lemma:mapsubspace} to $\LL^*$ and density parameter $\delta=\exp(-O(d^{5}))$, we obtain a subspace $V \subset \F^m$ of codimension $O(d^7)$ such that for each $y \in W$,
\begin{equation}\label{equation:final}
	\Pr_{\mathbf{s_1},\mathbf{s_2}\in S_y}\left[V\cap \overline{\LL}(y)^\perp\subset (\overline{\LL}(\mathbf{s_1})+\overline{\LL}(y))^\perp+(\overline{\LL}(\mathbf{s_2})+\overline{\LL}(y))^\perp\right]\geq \frac{1}{2},
\end{equation}
which implies
\begin{equation}\label{equation:afterfinal}
	\Pr_{\mathbf{s_1},\mathbf{s_2}\in (\F^n)^{11}}\left[V\cap  \overline{\LL}(y)^\perp \subset  \mathcal{V}_{y,\mathbf{s_1}}-_\eta \mathcal{V}_{y,\mathbf{s_2}} \right]\geq \exp(-O(d^{5})).
\end{equation}
Define the final bilinear structure as
$$
B^6:= \bigcup_{y\in W}\left( V \cap {\overline{\LL}(y)}^\perp \right) \times \{y\}.
$$
It satisfies
$$
B^6\subset \phi_{\hr\vr\vr\hr\vr}^{\eps_6}(B^2)
$$
with $\eps_6 = \exp(-O(d^{5}))$,
and so overall
$$B^6 \subset \phi_{\hr\vr\vr\hr\vr\vr\vr\hr\hr}^\eps(A)$$
with $\eps = \exp(-O(d^{5}))$.

\bibliographystyle{amsplain}


\begin{dajauthors}
\begin{authorinfo}[kaave]
  Kaave Hosseini\\
  University of California, San Diego\\
  USA\\
  skhossei\imageat{}ucsd\imagedot{}edu \\
  \url{http://cseweb.ucsd.edu/~skhossei/}
\end{authorinfo}
\begin{authorinfo}[shachar]
  Shachar Lovett\\
  University of California, San Diego\\
  USA\\
  slovett\imageat{}ucsd\imagedot{}edu \\
  \url{https://cseweb.ucsd.edu/~slovett/home.html}
\end{authorinfo}
\end{dajauthors}

\end{document}